\documentclass[lettersize,journal%,onecolumn
]{IEEEtran}
\usepackage{amsthm}
\newtheorem{theorem}{Theorem}

\newtheorem*{cor*}{Corollary}
\newtheorem{lem}[theorem]{Lemma}
\newtheorem{pro}[theorem]{Proposition}

\usepackage{comment}
\usepackage[utf8]{inputenc}
\usepackage[english]{babel}
\usepackage{eurosym}
\usepackage{xcolor}
\usepackage{mathrsfs}
\usepackage{graphicx}
\usepackage{float}
\usepackage{bbold}
\usepackage{bm}
\usepackage{mathtools}
\usepackage{algorithm}
\usepackage{alphabeta} % for displaying greek letters in the code snippets
\usepackage{amssymb,amsfonts, amsmath}
\usepackage{bm}
\usepackage{textcomp}
\usepackage[hyphens]{url} % enabling linebreaks in references
\usepackage[T1]{fontenc}
\newcommand{\Kappa}{{K}}
\usepackage{lmodern}
\usepackage{csquotes}
\usepackage{siunitx}
\usepackage{caption}
\usepackage{subcaption}
\usepackage[colorlinks=true, allcolors=blue]{hyperref}

\definecolor{green}{rgb}{0.,0.5,0.5}
\definecolor{dblue}{rgb}{0,0,1}

\newcommand{\zbar}{{\overline{z}}}

\newcommand{\sigmabar}{{\overline{\sigma}}}
\newcommand{\thetabar}{{\overline{\theta}}}

\newcommand{\Kappabar}{{\overline{\Kappa}}}
\newcommand{\ibar}{{\overline{\imath}}}
\newcommand{\vbar}{{\overline{v}}}

\newcommand{\Cbar}{{\overline{C}}}

\newcommand{\etabar}{{\overline{\eta}}}

\newcommand{\XTX}[1]{{\bm B}^\dagger{\bm {\mathcal{T}}}(#1){\bm B}}

\newcommand{\C}{\mathbb{C}}
\newcommand{\R}{\mathbb{R}}

\newcommand{\Tn}{{\bm T}}
\newcommand{\Te}{\bm{\mathcal{T}}\!\!}
\newcommand{\Tnet}{{\bm{\mathcal{T}}^{\rm net}}}
\newcommand{\Tnod}{{\bm T}^{\rm nod}}

\newcommand{\bmat}[1]{\begin{bmatrix} #1 \end{bmatrix}}

\usepackage{algorithmic}
\usepackage{array}
\usepackage{stfloats}
\begin{document}

\title{Small-signal stability of power systems with voltage droop}
% interesting key-words we could use:
% smart grids
% heterogeneous
% decentralized
% technology-neutral
% do we need "voltage" or "control" in voltage droop control? One could be removed for brevity.

% old titles:
% Small signal stability conditions for power systems with voltage droop control
% Phase Stability of Inverter-Based Power Systems with Voltage Droop Control}

\author{Jakob Niehues,~\IEEEmembership{Graduate Student Member,~IEEE,}
        % <-this % stops a space
\thanks{Corresponding author: 
\href{mailto:jakob.niehues@pik-potsdam.de}{jakob.niehues@pik-potsdam.de}}
\thanks{Potsdam Institute for Climate Impact Research (PIK),
 Member of the Leibniz Association, P.O. Box 60 12 03, D-14412 Potsdam, Germany (J.N., A.B., F.H.)}
 \thanks{
Technische Universit\"at Berlin, ER 3-2, Hardenbergstrasse 36a, 10623 Berlin, Germany (J.N.)}
Robin Delabays,~\IEEEmembership{Non-Member,~IEEE,}
\thanks{School of Engineering, University of Applied Sciences of Western Switzerland HES-SO, Sion, Switzerland (R.D.)}
Anna Büttner,~\IEEEmembership{Member,~IEEE,}
Frank Hellmann,~\IEEEmembership{Member,~IEEE}
}

\maketitle

\begin{abstract}
The stability of inverter-dominated power grids remains an active area of research. This paper presents novel sufficient conditions for ensuring small-signal stability in lossless and constant $R/X$ grids with highly heterogeneous mixes of grid-forming inverters that implement an adapted $V$–$q$ droop control. The proposed conditions can be evaluated in the neighborhood of each bus without information on the rest of the grid.  Apart from the presence of $V$–$q$ droop, no additional assumptions are made regarding the inverter control strategies, nor is dynamical homogeneity across the system assumed. The analysis is enabled by recasting the node dynamics in terms of complex frequency and power, resulting in transfer functions that directly capture the small-signal frequency and amplitude responses to active and reactive power imbalances. These transfer functions are directly aligned with typical design considerations in grid-forming control. Building on an adapted small-phase theorem and viewing the system as a closed feedback loop between nodes and lines, the derived stability conditions also yield new insights when applied to established inverter control designs. We demonstrate in simulations that our conditions are not overly conservative and can identify individual inverters that are misconfigured and cause instability.

\end{abstract}


\section{Introduction}

The analysis of the small-signal stability of multi-machine power grids is one of the central topics of power grid analysis. The main result of the seminal paper of \cite{bergen_structure_1981} was to give conditions under which multiple machines and loads, modeled as oscillators, are stable to small perturbations.

Since then, a plethora of results from power engineering \cite{machowski_power_2012}, control theory \cite{dorfler_synchronization_2013, yang_distributed_2020, schiffer_conditions_2014} and theoretical physics \cite{witthaut_collective_2022, bottcher_stability_2023} have expanded our understanding of the small signal stability of power systems. However, it remains an active topic of research \cite{woolcock_mixed_2023, pates_scalable_2017, persis_bregman_2016, huang_gain_2024}.
In recent years, the topic has gained renewed interest with the introduction of grid-forming converters, which are expected to independently stabilize the synchronous operation of highly renewable future power grids \cite{christensen_high_2020}. Grid-forming control remains an active topic of research, and additionally, often detailed device models are not published by the vendor \cite{dysko_testing_nodate,chen_generalized_2022,haberle_grid-forming_2024}.
There is a wide range of stability results for concrete control strategies, as reviewed in \cite{he_quantitative_2024}. 
However, most of them are {\it ad hoc} and do not generalize naturally to other control schemes.

In this paper we give a decentralized stability condition based on the transfer functions that describe how a grid-forming node's frequency and relative voltage velocity react to deviations from power, reactive power and voltage set points. 
Remarkably, our results are technology-neutral and apply to all grid-forming nodal actors for which the response to reactive power and voltage set point deviations is proportional, which is an established principle, see for example \cite{schiffer_survey_2016, schmietendorf_self-organized_2014}.

The variables used in this work correspond to working with the complex frequency \cite{milano_complex_2022} and describing the network state using time-invariant variables that nevertheless fully characterize the operating state at the desired frequency \cite{kogler_normal_2022, buttner_complex_2024}. Such variables have been shown to be highly effective for identifying grid-forming behavior in the grid \cite{buttner_complex-phase_2024}. As we will see, an advantage of working in these quantities is that the transfer matrices do not depend on arbitrary quantities such as phase angles. The resulting stability conditions are more explicit, simpler and more easily interpreted than, for example, those of \cite{huang_gain_2024, woolcock_mixed_2023, yang_distributed_2020}. In particular, the transfer matrices often do not explicitly depend on the operation point around which we linearize, and the conditions can be mapped back to system parameters immediately. We demonstrate this by recovering several classical results as special cases.

As in \cite{huang_gain_2024, woolcock_mixed_2023}, the central ingredient to our result is the small phase theory introduced in \cite{chen_phase_2024}. A companion paper to this work \cite{kastendiek_phase_2025} explores the application of this approach to the broad class of adaptive dynamical networks \cite{berner_adaptive_2023} and demonstrates that these methods can match necessary conditions in that setting. This approach can be seen as an extensive generalization of passivity. Passivity-based methods have been previously used to derive decentral stability conditions for scalar networked systems \cite{lestas_scalable_2007} and for power grids \cite{yang_distributed_2020, he_passivity_2023}. We improve on these results by giving more broadly applicable conditions that are fully decentralized and less conservative. Similar results were independently obtained in \cite{haberle_decentralized_2025}, however, only for a heavily restricted class of models when compared to our results.

%%%%%%%%%%%%%%%%%%%%%%%%%%%%%%%%%%%%%%%
\section{Statement of the main result}
We begin by presenting the key assumptions and the main result, using the bare minimum of notation and concepts necessary to state them. For clarity, we first treat lossless systems. The case of homogeneous ratio of resistance to reactance is treated in section \ref{sec:lossy lines}.

We assume a lossless grid with admittance $\bm Y$, which is a Laplace matrix.
Denote nodal complex voltages $\bm v = \bm v_d + j \bm v_q$, a vector with components
\begin{align}
    v_n(t) = V_n(t) e^{j \varphi_n(t)} \; ,
\end{align}
with phase $\varphi_n$ and amplitude $V_n$.
The nodal current injections are $\bm \imath = \bm Y \bm v$, and the nodal power injections $p_n+j q_n = v_n \ibar_n$. 

Quantities at the operating point are written with a superscript~$^\circ$.
In the co-rotating frame with the grid's nominal frequency, the operating point is given by constant $v_n^\circ$ that induce $V_n^\circ$, $\varphi_n^\circ$, and a power flow solution $p_n^\circ$, $q_n^\circ$ matching the set point.

We assume that the dynamics of the nodes can be formulated in terms of the complex frequency $\eta_n:=\dot v_n / v_n$ (see~\cite{milano_complex_2022,kogler_normal_2022} for details).
Its real part $\varrho_n = \dot V_n/{V_n} $ is the relative amplitude velocity, and its imaginary part, $\omega_n = \dot \varphi_n$, is the angular velocity, which is proportional to the frequency. Without loss of generality, we take the complex frequency at the operational state to be equal to zero: $\omega^\circ = \varrho^\circ = 0$. In practical terms, this assumption implies that all nodes have some amount of grid-forming capability.

We can understand the behavior of a broad class of dynamical actors in power grids by considering how their complex frequency reacts to changes in the network state. Near the power flow solution of interest, we can consider the linearized response in terms of the transfer functions.
From this perspective, grid-forming actors take the current as input and supply a voltage as output. We will focus our analysis on systems that implement a droop relationship between voltage and reactive power. This droop relationship is typical in models of power grid actors \cite{schmietendorf_self-organized_2014,schiffer_conditions_2014}. We will use $p$ and the shifted reactive power $\hat q_n \coloneqq q_n + \alpha_n V_n$ that implements the $V$-$q$ droop relationship with proportionality coefficient $\alpha_n \in \mathbb{R}$ as input for the nodes.

We then have four transfer functions $T_n^{\bullet\bullet}(s)\in\mathbb{C}$ that describe the nodal behavior near the power flow of interest:
\begin{align}
 \bmat{\varrho_n \\ \omega_n} = -\bmat{T_n^{\varrho \hat q} & T_n^{\varrho p} \\ T_n^{\omega \hat q} & T_n^{\omega p} } \bmat{\Delta \hat q_n \\ \Delta p_n} =: -\Tn_n \bmat{\Delta \hat q_n \\ \Delta p_n}\, ,
 \label{eq:definition_entries_T_n}
\end{align}
where all quantities except $\alpha_n$ depend on the Laplace frequency $s$.

Following \cite{kogler_normal_2022}, the matrix elements of $\Tn_n(s)$ are expected to only depend on $p^\circ$, $q^\circ$ and $V^\circ$, but not on the complex voltage $v_n^\circ$ directly. As $v_n^\circ$ is only defined uniquely up to phase, this is a key advantage of working in terms of phase shift invariant quantities like $p$, $q$ and $\eta$ rather than, say, $\dot v$, $\vbar$, and $\imath$, $\ibar$. This mirrors the choice of power and polar coordinates in \cite{yang_distributed_2020}.
Our main result is:

\begin{pro}[Small-signal stability of power grids with $V$-$q$ droop]
\label{pro:main:power_grid}
    Consider a lossless power grid with admittance matrix $\bm Y$ and an operating point with voltage phase angles $\varphi^\circ_n$ and magnitudes $V_n^\circ$, and $ \bm T_n(s)$ the transfer function matrices from $\hat q_n$, $p_n$, to $\varrho_n$ and $\omega_n$ for some $\alpha_n$.
    
    The operating point is linearly stable if $|\varphi_n^\circ - \varphi_m^\circ| < \pi /2$ for all $n$ and $m$ connected by a line, the $\Tn_n(s)$ are internally stable, and for all $s\in[0,\infty]$ it holds
\begin{align}
        \Re(T_n^{\varrho \hat q}) + \Re(T_n^{\omega p}) & > 0\, ,
        \label{eq:tr_T_nodes_definite}
        \\
        \Re(T_n^{\varrho \hat q}) \cdot \Re(T_n^{\omega p}) &> \frac{1}{4}\left|T_n^{\varrho p} + \overline T_n^{\omega \hat q}\right|^2\, ,
        \label{eq:det_T_nodes_positive}
        \\
        \alpha_n & \ge 2 \sum_m \tilde Y_{nm}  \frac{V_m^\circ}{\cos(\varphi^\circ_n - \varphi^\circ_m)} \, .\label{eq:alpha_bound_T_lines_definite}
\end{align}
   
\end{pro}

\begin{proof}
    We provide the proof in appendix \ref{sec:Proof: Linear stability of power grids with $V$-$q$ droop}.
\end{proof}
We restrict our analysis to systems for which there is a choice of $\alpha_n$ that eliminates $V_n$ as a nodal state variable by absorbing it into $\hat q$. Otherwise, the first two conditions might fail for small $s$. The reason for this is that $V_n$ is a local state variable at the bus, while $\dot V_n$ appears as output. This is in contrast to $\varphi_n$, which does not appear \cite{kogler_normal_2022}. This mismatch makes the Hermitian part of the transfer function matrix non-definite for small $s$. Choosing $\alpha_n$ such that it eliminates $V_n$ as a nodal state variable makes $\Tn_n$ well-behaved. This can easily be achieved for many models of power grid actors \cite{schmietendorf_self-organized_2014,schiffer_conditions_2014} and notably also covers all systems analyzed in \cite{yang_distributed_2020}. The precise model class is discussed in more detail in appendix \ref{app:linear form of power grids with v-q droop}. From here on we assume that $\alpha_n$ is chosen in this way.
An alternative approach is to restrict the model class such that the transfer function matrix remains well behaved, e.g. by requiring $T_n^{\varrho p} = T_n^{\omega \hat q} = 0$. This alternative approach has been explored independently in depth in \cite{haberle_decentralized_2025}.

Our conditions align well with established practice in the design of grid-forming power grid actors. The diagonal terms $T_n^{\varrho \hat q}$ and $T_n^{\omega p}$ implement a stabilizing reaction of phase and amplitude to active and reactive power deviations, respectively.
Equations \eqref{eq:tr_T_nodes_definite}-\eqref{eq:det_T_nodes_positive} together imply that these transfer functions need to have negative real parts and dominate the dynamics.
In addition, \eqref{eq:det_T_nodes_positive} quantifies how large the crosstalks $T_n^{\omega \hat q}$ between reactive power and frequency, and $T_n^{\varrho p}$ between active power and voltage amplitude, may be, without endangering stability.

From the physics of the interconnection, we get a third condition: 
that the stabilization of the amplitude is sufficiently strong relative to the coupling on the network, as quantified in \eqref{eq:alpha_bound_T_lines_definite}.
This condition relates the nodal $V$-$q$ droop ratio $\alpha_n$ to local grid conditions. Note in particular that the lower bound in \eqref{eq:alpha_bound_T_lines_definite} can be negative, indicating that local grid conditions are so strong that even misconfigured droop relationships can be tolerated.

The remainder of this paper is structured as follows.
In Section \ref{sec:examples}, we derive what our main results imply in concrete systems and compare them with the results of \cite{yang_distributed_2020} and \cite{bottcher_stability_2023}. We then present numerical results for the IEEE 14-bus system in Section~\ref{sec:simulations}, which demonstrate that our conditions can be tight in this setting. Finally, we present the generalization to lossy grids in Section~\ref{sec:lossy lines} and provide a discussion and outlook in Section \ref{sec:discussion and conclusion}. The appendix includes the relevant mathematical definitions, derivations, and proofs.

\section{Concrete systems}
\label{sec:examples}

We will now demonstrate that the conditions of Proposition \ref{pro:main:power_grid} are viable to study the behavior of a wide range of typically considered grid models, and often can even improve on established theoretical considerations. We begin with generalized droop laws.

\subsection{Generalized droop}

The most general dynamical droop law relating voltage, frequency, active and reactive power is of the form:
\begin{align}
    \dot\varphi &= c_1 \Delta p + c_2 \Delta q + c_3 \Delta V,
    \\
    \dot V &= c_4 \Delta p + c_5 \Delta q + c_6 \Delta V.
\end{align}

Our assumption on exact droop behavior implies $c_6/c_5 = c_3/c_2 =: \alpha$, and we can reparametrize this as
\begin{align}
    \dot\varphi &= -C^\omega_{p}  \Delta p - C^\omega_q \Delta \hat q,
    \label{eq:second_order_droop_phi}
    \\
    \dot V &= V^\circ \cdot \left( -C^V_p \Delta p - C^V_{q} \Delta \hat q \right).
    \label{eq:second_order_droop_V}
\end{align}

This is also the most general form that the linearized equations of a grid forming device with exact $V$-$q$ droop can take when neglecting internal dynamics \cite{kogler_normal_2022}. The class of models considered in \cite{schiffer_conditions_2014} and \cite{yang_distributed_2020} Proposition 5 and 6 is a special case of the class studied in this section.

In this section, we discuss and contrast the theoretical results. Below, in Section \ref{sec:simulations}, we will show that our conditions are also remarkably exact in this model class.

The transfer matrix for \eqref{eq:second_order_droop_phi}, \eqref{eq:second_order_droop_V} is
\begin{align}
 \pmb T_n(s) = \begin{bmatrix} C^V_q &  C^V_p \\ C^\omega_q & C^\omega_p \end{bmatrix} 
\end{align}
and \eqref{eq:tr_T_nodes_definite}-\eqref{eq:det_T_nodes_positive} become
\begin{align}
        C^V_{q} + C^\omega_{p} & > 0
        \label{eq:tr_T_nodes_definite_real}
        \\
         C^V_{q} \cdot C^\omega_{p} &> \frac{1}{4}\left(C^V_{p} + C^\omega_{q}\right)^2.
        \label{eq:det_T_nodes_positive_real}
\end{align}

The well-established droop principles of controlling $\varphi_n$ with $-\Delta p_n$ and $V_n$ with $-\Delta q_n$ and $-\Delta V_n$ (see for example \cite{schiffer_survey_2016, schmietendorf_self-organized_2014}) are reflected in $T_n^{\omega p}>0$ and $T_n^{\varrho \hat q} > 0$.
Equations~\eqref{eq:tr_T_nodes_definite_real}-\eqref{eq:det_T_nodes_positive_real} tell us that these coefficients need to have the same sign and need to be positive.
Equation~\eqref{eq:det_T_nodes_positive_real} further quantifies that cross-coupling, reflected by $T_n^{\varrho p}$ and $T_n^{\omega \hat q}$, needs to be sufficiently small in comparison.

The case considered in \cite{yang_distributed_2020} Proposition 5 corresponds to $C^V_p = C^\omega_q = 0$. Then our stability conditions simplify to $C^V_{q} > 0 $ and $ C^\omega_{p} > 0$ together with the condition on $\alpha$. The conditions presented here improve upon those in Proposition 5 of \cite{yang_distributed_2020} for this model class. They require that $C^\omega_p$ and $\alpha_n$ are larger than a positive constant that depends on the entire network, and assume the signs of $C^V_{q}$ and $C^\omega_{p}$ from the outset. In contrast, we find no bound other than the `sign' on the $C$, and our lower bound for $\alpha_n$ is a local quantity that can even become negative. We will illustrate that this occurs in practical grid situations in the section on numerical experiments.

\subsection{Third-order models}

We now compare our results to established conditions in the widely studied case of second-order phase dynamics and voltage control.
For this purpose, we need a single internal variable $x_n$ that represents the phase velocity (angular frequency) relative to the nominal frequency.
For purposes of regularization, we further introduce a first-order feed-through term with coefficient $\delta_n$:
\begin{align}
    \dot\varphi_n &= x_n - \delta_n \Delta p_n\,\label{eq:Schiffer_phi_dot} ,\\
    \tau_{p_n} \dot x_n &= - D_n x_n - k_{p_n} \Delta p_n\, ,\\
    \tau_{q_n} \dot V_n &= - \Delta V_n  - k_{q_n} \Delta q_n\, .
\end{align}
At $\delta_n = 0$ we have pure second-order phase dynamics.
We adapted the notation of the droop-controlled inverter model of \cite{schiffer_conditions_2014}, which we recover at $\delta_n = 0$.
With $k_{q_n} = \alpha_n^{-1}$, the transfer matrix is given by
\begin{align}
    \Tn_n = \bmat{(V_n^\circ \alpha_n \tau_{q_n})^{-1} & 0 \\ 0 & \delta_n +\frac{k_{p_n}}{s \tau_{p_n}+D_n}}\, ,
\end{align}
assuming $\tau_{p_n}>0$ and $\tau_{q_n} > 0$.
A similar model is the third-order model for synchronous machines \cite{schmietendorf_self-organized_2014}, where the voltage dynamics are slightly different:
\begin{align}
    \tau_{V_n} \dot V_n = -\Delta V_n - X_n \Delta(q_n/V_n)\, ,
\end{align}
with transient reactance $X_n \ge 0$.
The transfer matrices of both models are identical via the invertible mapping
\begin{align}
    X_n &= V_n^\circ k_{q_n} \left(1 + 2 \frac{k_{q_n} q_n^\circ}{V_n^\circ}\right)^{-1}\, ,
    \\
    \tau_{V_n} &= \tau_{q_n} \left(1 + 2 \frac{k_{q_n} q_n^\circ}{V_n^\circ}\right)^{-1}\, .
\end{align}
This transfer matrix also represents the dynamics of virtual synchronous machines \cite{shuai_transient_2019}, quadratic droop control \cite{simpson-porco_voltage_2017}, reactive current control \cite{persis_bregman_2016}, and some controls with adaptive inertia \cite{fritzsch_stabilizing_2024} through similar mappings.

For the nodal transfer matrices to be stable as required by Proposition~\ref{pro:main:power_grid}, we need $D_n > 0$.
Conditions \eqref{eq:tr_T_nodes_definite}-\eqref{eq:det_T_nodes_positive} are fulfilled at all $s$ as long as $\delta_n > 0$, $k_{p_n} > -\delta_n D_n$ and $\alpha_n > 0$.

At $\delta_n = 0$ and $s=j\infty$, we have $T_n^{\omega p} = 0$ and violate \eqref{eq:det_T_nodes_positive}.
However, this is sufficient to establish semi-stability at $\delta_n = 0$, because
stability holds for arbitrarily small $\delta_n$ and the eigenvalues of the system's Jacobian are continuous functions of the parameters.
Furthermore, including gain information allows to treat this system at $\delta_n=0$, too \cite{kastendiek_phase_2025}.

In \cite{schiffer_conditions_2014}, stability conditions for this model were given in terms of matrix inequalities with a similar interpretation to our analysis: the diagonal couplings $T_n^{\varrho\hat q}$ and $T_n^{\omega p}$ need to be strong in the positive direction, while the off-diagonal cross-coupling need to be bounded relatively. We obtain a similar result, which, however, is decentralized and thus easier to analyze and implement.

In \cite{schiffer_conditions_2014}, it was also observed that decreasing $k_{q_n}$ can increase stability by weakening the cross-coupling. This is quantified in our lower bound for $\alpha_n = k_{q_n}^{-1}$ in \eqref{eq:alpha_bound_T_lines_definite}:
\begin{align}
    k_{q_n}^{-1} \ge  2 \sum_{m} \tilde Y_{nm}\frac{V_m^\circ}{\cos(\varphi_n^\circ - \varphi_m^\circ)}\, .
\end{align}
To our knowledge, this lower bound is entirely novel and has not previously been reported in the literature.
In \cite{bottcher_dynamic_2022,bottcher_stability_2023}, under the assumption that all nodes in the system are of the same functional form, a bound for $k_{q_n}$ was also derived. This bound can be tighter or looser than ours, depending on the operating points.

\section{Simulations}
\label{sec:simulations}
To test our stability conditions, we simulated the stability of the IEEE 14-bus system equipped with grid-forming inverters following the generalized droop control given in equations \eqref{eq:second_order_droop_phi}-\eqref{eq:second_order_droop_V}.

First, we tested the condition given in equation \eqref{eq:alpha_bound_T_lines_definite}. We stressed the grid by simulating imperfect reactive power provision. We choose reactive power values varying by a random factor of $\pm 0.3$ around the ideal reactive power per node, leading to grid states with moderate voltage variation. We then computed the sufficient bounds for $\alpha$ to be stable, $\alpha_n^\text{theory}$. Setting all inverters to $\alpha_n^\text{theory}$, we then systematically varied one inverter setting to find the critical value $\alpha_n^\text{crit}$ necessary for stability. If our conditions are overly conservative, we would expect that $\alpha_n^\text{crit}$ is smaller than $\alpha_n^\text{theory}$. Instead, we see in Figure~\ref{fig:alpha_test} that the theoretical prediction is almost perfect. We also observe that $\alpha_n$ can locally be negative. This demonstrates the power of our theoretical analysis to take into account local grid conditions in a far more sophisticated manner than previous analyses. In fact, in the system tested, our theoretical predictions only vary notably from the simulation results when $\alpha_n^\text{crit}$ is very close to zero. In this case, our prediction becomes slightly conservative (Figure~\ref{fig:alpha_test} inset).

\begin{figure}
    \centering
    \includegraphics[width=\linewidth]{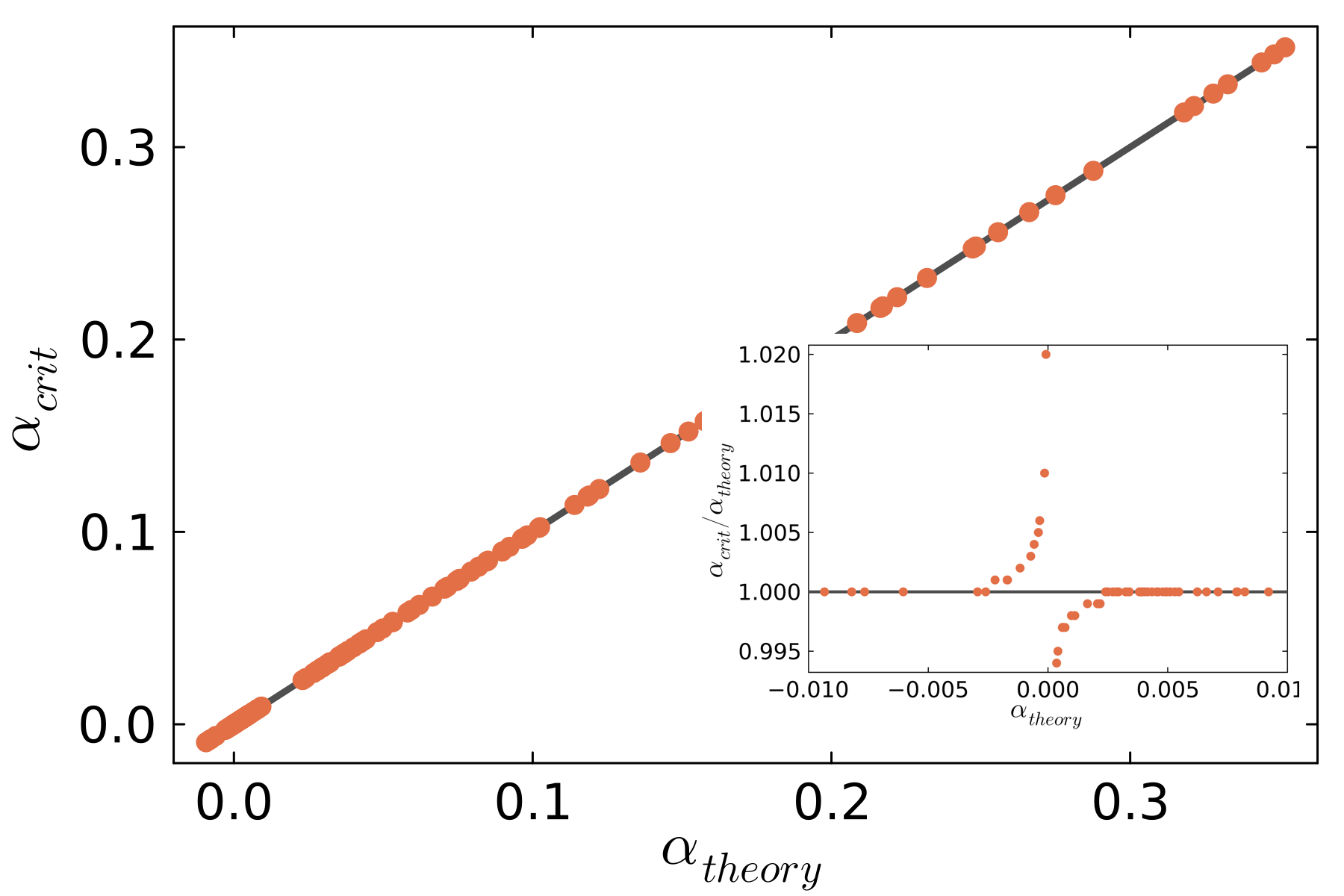}
    \caption{Numerical small-signal stability of the IEEE 14-bus system for $C_p^V = 1$, $C_q^V = 1$, $C_q^V = 0.5$, $C_p^\omega = 0.5$. The predicted $\alpha^\text{theory}$ exactly matches the numerically simulated stability threshold except when close to $0$ (inset). When $\alpha_n^\text{crit} < \alpha_n^\text{theory}$, we observe $\alpha_n^\text{crit} / \alpha_n^\text{theory} < 1$ for positive $\alpha$ and $\alpha_n^\text{crit} / \alpha_n^\text{theory} > 1$ for negative $\alpha$.}
    \label{fig:alpha_test}
\end{figure}

Figure~\ref{fig:trajectory} illustrates example trajectories for a stable system where all nodes are at the theoretical $\alpha$ value, and an unstable system where one node violates the theoretical stability guarantee. We observe that the violation leads to a slow voltage collapse within the system. As only one node violates our theoretical bound in this system, our bounds successfully pinpoint the origin of instability in this case.

\begin{figure}
    \centering
    \includegraphics[width=0.85\linewidth]{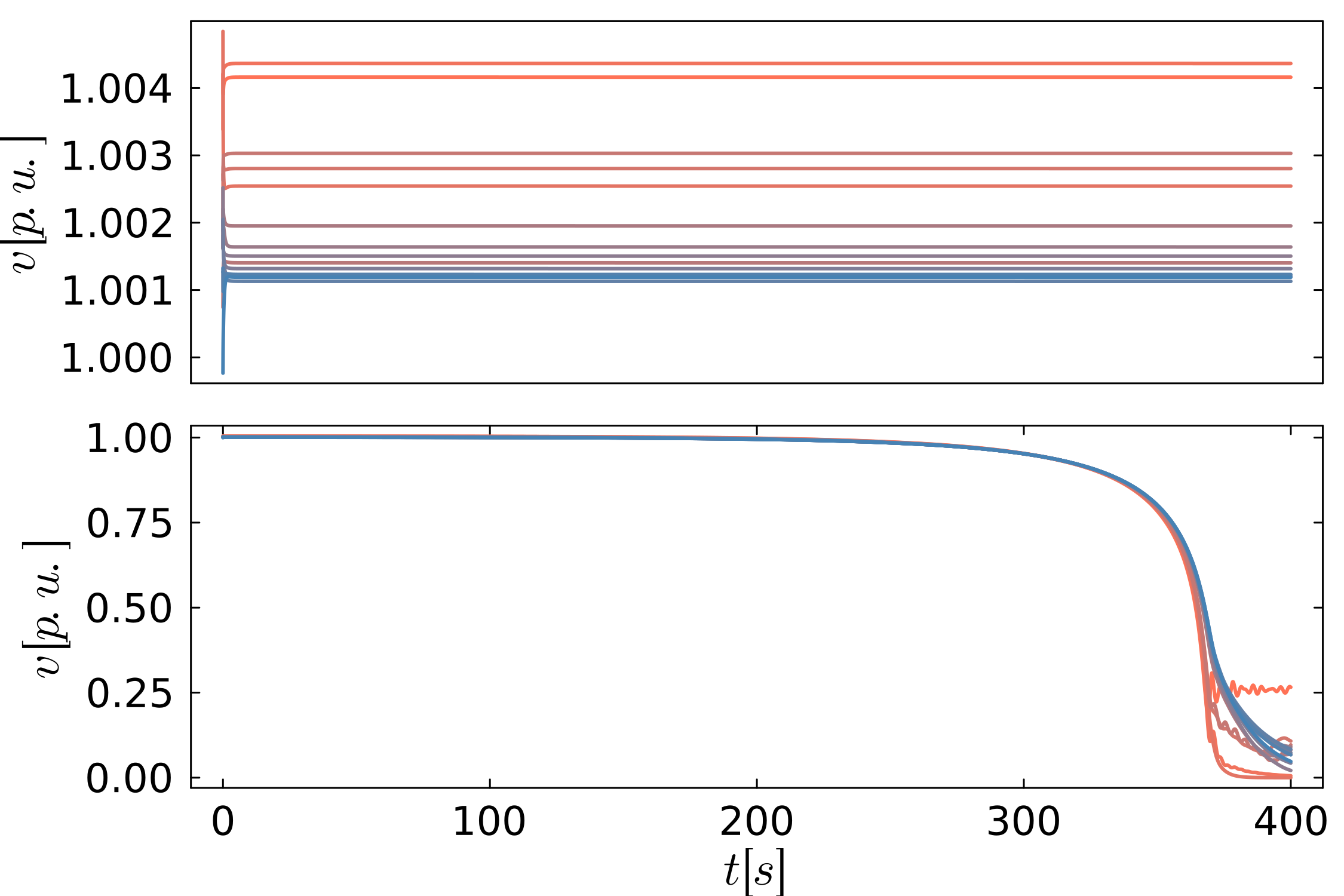}
    \caption{Trajectories for $\alpha_n = \alpha_n^\text{theory}$ (upper) and the case where $\alpha_1 < \alpha_n^\text{theory}$. Improper configured voltage droop at one node causes a slow voltage collapse.}
    \label{fig:trajectory}
\end{figure}

Finally, we also tested condition \eqref{eq:det_T_nodes_positive} for the same setup. We set $\alpha = \alpha^{\text{theory}}_n$ at each node, and varied the strength of the cross-coupling terms $C_p^V$ and $C_q^\omega$, which couple active power to voltage amplitude and reactive power to frequency, respectively. We observed that when the cross-couplings are of similar magnitude to the main couplings $C_q^V$ and $C_p^\omega$, which is to be expected for a well-tuned inverter, our stability conditions accurately capture the boundary of stability (Figure~\ref{fig:stability_matrix}). The conditions become conservative only when one cross-coupling term remains small while the other becomes large, which is an untypical control setting.

\begin{figure}
    \centering
    \includegraphics[width=\linewidth]{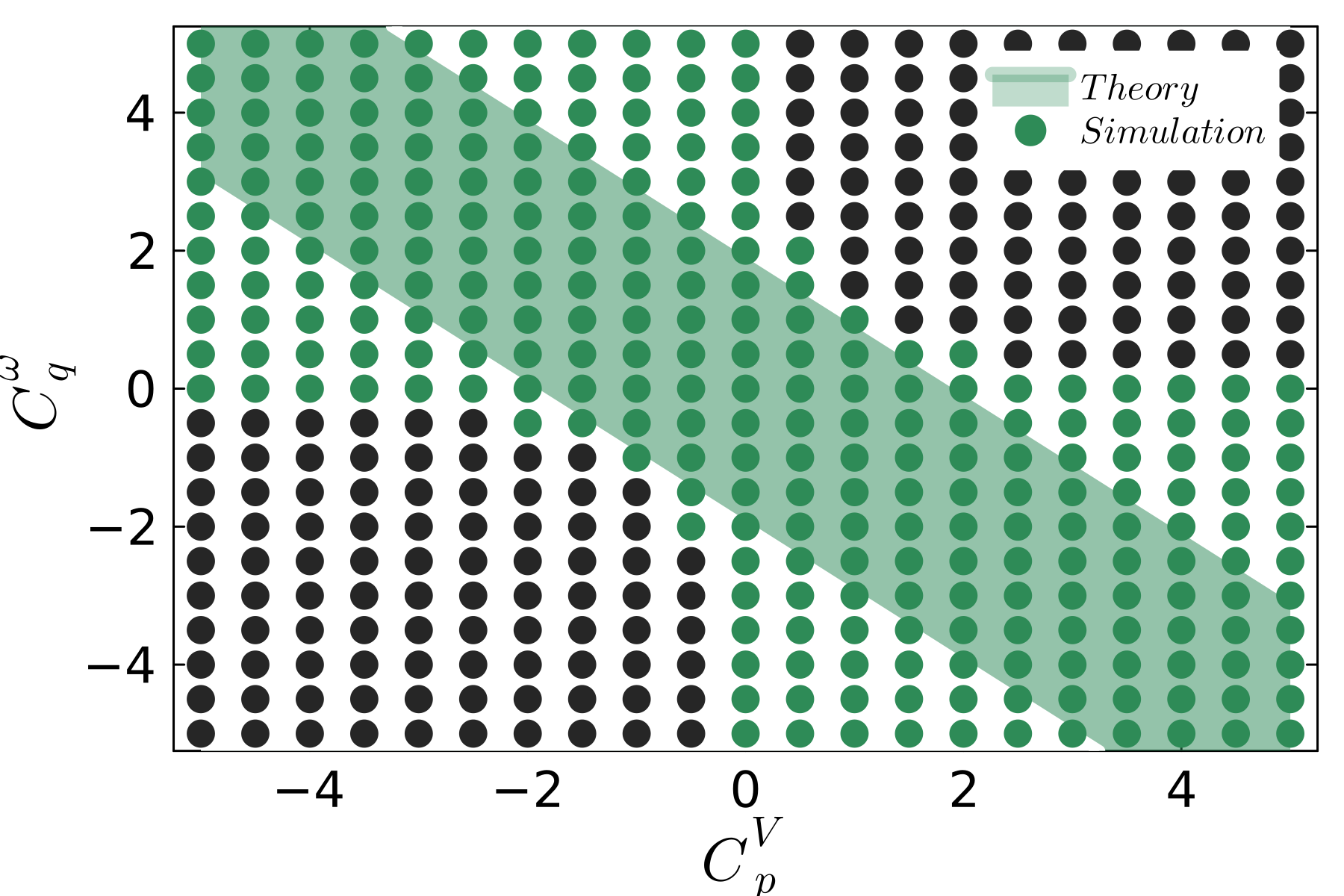}
    \caption{Numerical small-signal stability of the IEEE 14-bus system for $C_q^V = 1$, $C_p^\omega = 1$, and $\alpha^{\text{theory}}_n$. Green and black dots indicate numerical linear stability and instability, respectively, for each parameter configuration. The shaded area indicates our sufficient condition. In the case that $C_q^V = C_p^\omega$, our condition is exact.}
    \label{fig:stability_matrix}
\end{figure}

%%%%%%%%%%
\section{Lossy lines}
\label{sec:lossy lines}
The principles of controlling $V$ with $\hat q$ and $\varphi$ with $p$, which are quantified by Proposition \ref{pro:main:power_grid}, are valid for lossless transmission lines.
In the presence of losses, similar principles hold with $\hat q$ and $p$ getting mixed depending on the ratio of resistance $R$ and reactance $X$.

Assuming constant $R/X$ ratio for all lines, we define $\tan\kappa:=R/X$.
The rescaled rotation matrix $\bm O(\kappa)$, and rotated transfer function matrix $\tilde{\bm T}_n$ are defined as
\begin{align}
    \bm O(R/X) &:= \bmat{1 & -R/X\\
    R/X & 1},
    \\
    \bm \Tn_n(s) &:= \tilde{\bm T}_n(s)\; \bm O(R/X).
\end{align}
Note that $\bm O \cos\kappa$ is a rotation matrix.
In the lossless case, we have $\bm O=\bm I$, the identity, and $\tilde{\bm \Tn}_n = \bm \Tn_n$.
The nodes now obey
\begin{align}
    \bmat{\varrho_n\\\omega_n} = - \tilde{\bm \Tn}_n(s) \left(\bmat{\Delta q_n\\ \Delta p_n}
    + \bm O \bmat{\alpha_n \Delta V_n\\0} \right).
\end{align}
This way, the conditions of Proposition \ref{pro:main:power_grid} for $\bm \Tn_n$ and $\alpha_n$ also hold for lossy grids, with an analogous proof, because the admittance can be rotated real for the analysis of the transmission lines' transfer matrix.

What does this parametrization mean in practice?
To interpret the conditions on $\bm \Tn_n$, consider that it can be seen as a transfer function from the lines' output
\begin{align}
    \bm O^{-1} \bmat{\Delta q_n\\ \Delta p_n} + \bmat{\alpha_n \Delta V_n\\0}
\end{align}
to $\bmat{\rho_n & \omega_n}^\top$.
This is a droop between $\hat q = q + \alpha V$ as before, and $\hat p = p + \alpha V R/X$ instead of just $p$, i.e., the control is adapted to the $R/X$ ratio. This mirrors the control design considered, for example, in \cite{he_passivity_2023}, where current and power are also rotated by the angle defined by $R/X$.

\section{Discussion and Conclusion}\label{sec:discussion and conclusion}
\label{discussion and conclusion}
In this paper, we derived fully decentralized small-signal stability conditions for power grids under the assumption of $V$-$q$ droop and homogeneous $R/X$ ratio for the lines. The preceding results provide a simple characterization of small-signal stability of heterogeneous grids in terms of transfer matrices between power mismatch on the input side, and frequency and voltage velocity on the output side. Such transfer function-based specifications are natural for the design and specification of decentralized power grid control strategies, and could potentially be directly encoded in grid codes \cite{haberle_dynamic_2023}. This is especially interesting as the transfer functions we are concerned with can be measured experimentally \cite{buttner_complex-phase_2024}.

The type of conditions derived here are robust in the sense that, if the numerical range of a nodal transfer matrix is bounded away from zero for all $s$ on the contour (see proof), a perturbation of the transfer matrix of $H_\infty$ norm smaller than the bound, can not make the system unstable. However, as we have to assume an exact droop relationship, this robustness does not yet easily extend to actual system parameters.

As the complex frequency approach \cite{milano_complex_2022} can also capture load models and grid-following control, as shown in \cite{moutevelis_taxonomy_2024}, we expect that our results can be adapted to load models. A starting point for an extension to line dynamics is given in \cite{haberle_decentralized_2025}. An alternative approach would be to use the observation in Appendix B of \cite{buttner_complex_2024} that line dynamics in the case of a homogeneous R/X ratio are essentially a low-pass filter on the nodal power flow that can be absorbed into the node dynamics.

The most significant challenge for our approach is to accurately account for non-droop-like reactions to voltage amplitude deviations. This also prevents us from directly applying the theory to conventional models in the presence of losses. Naively adding in additional voltage dynamics on the nodal side fails due to the sectoriality constraints. Similarly, models that do not have a pass-through like $\delta_n$ in \eqref{eq:Schiffer_phi_dot} fail our conditions at infinite imaginary $s$.
Lastly, dVOC \cite{seo_dispatchable_2019,gros_effect_2019} is covered by our theorem only in the unloaded case. To address these limitations, it will be necessary to accurately incorporate gain information into the stability analysis. The companion paper \cite{kastendiek_phase_2025}
explores this in the context of adaptive dynamical networks. We leave this extension of the methods introduced in this paper to future work.

\section*{Acknowledgments}
This work was supported by the OpPoDyn Project, Federal Ministry for Economic Affairs and Climate Action (FKZ:03EI1071A).

J.N. gratefully acknowledges support by \mbox{BIMoS} (TU Berlin), Studienstiftung des Deutschen Volkes, and the Berlin Mathematical School, funded by the Deutsche Forschungsgemeinschaft (DFG, German Research Foundation) Germany's Excellence Strategy –-- The Berlin Mathematics Research Center MATH+ (EXC-2046/1, project ID: 390685689).

R.D. was supported by the Swiss National Science Foundation under grant nr.~200021\_215336.

\section*{Conflict of interest}
None of the authors have a conflict of interest to disclose.

%%%%%%%%%%%%%%%%%%%%%%%%%%%%%%%%%%%%%%%%%%%%
{\appendices

\section{Notational Preliminaries}
\label{sec:notation}

To prove the above result, we begin by expanding on the notation used above. We want to consider the small-signal stability of power grids with a heterogeneous mix of grid-forming actors. The $N$ nodes %in $\mathcal{N}$ 
are indexed $n$ and $m$, $1 \leq n,m \leq N$.
The $E$ edges in the set of edges $\mathcal{E}$ are indexed by ordered pairs $e=(n,m)$, $n < m$. For any nodal quantity $x_n$, we denote the overall $N$-dimensional vector by $\bm{x}$. We write $[\bm x]$ for the diagonal matrix with $x_n$ on the diagonal: $[\bm x]_{nm} = \delta_{nm} x_n$, where $\delta_{nm} = 1 $if $n=m$, and $0$ else. In general, matrices are uppercase bold, e.g., $\bm A$, and vectors are lower case bold. We denote with $\bm 1$ the constant vector $1_n = 1$, so the identity matrix is $[\bm 1] =  \bm I$, and similarly for $\bm 0$ and $[\bm 0]$.

We denote the imaginary unit $j$, the complex conjugate of a quantity $z$ by $\zbar$, the transpose of a vector or matrix $\bm A$ as $\bm A^\intercal$ and the complex transpose by $\bm A^\dagger$.

We will often have two quantities per node, e.g., $z_n$ and $\zbar_n$. Stacking the vector of nodal quantities is written as
\begin{align}
\begin{bmatrix}
\bm z \\
\bm \zbar
\end{bmatrix}\, ,
\end{align}

We also will often be looking only at the components associated to a single node $n$ in such a stacked vector. To this end, we introduce the matrix $\bm P_n$ which selects these entries
\begin{align}
\bmat{z_n \\ \zbar_n} = \bm P_n \begin{bmatrix}
\bm z \\
\bm \zbar
\end{bmatrix}\, ,
\end{align}

and its transpose $\bm P_n^\dagger$. Note that $\bm P_n$ are isometries, and $\bm P_n^\dagger \bm P_n$ is an orthogonal projection matrix.

Given a set of nodewise matrices $\bm A_n$, the matrix built from them with the direct sum $\bigoplus$ then acts on our stacked vector as:
\begin{align}
\bigoplus_n \bm A_n \begin{bmatrix}
\bm z \\
\bm \zbar
\end{bmatrix} \coloneqq \sum_n \bm P_n^\dagger \bm A_n \bm P_n \begin{bmatrix}
\bm z \\
\bm \zbar
\end{bmatrix}\, ,
\end{align}

While the matrix representation of $\bigoplus_n \bm A_n$ is not block diagonal on the stacking $\bmat{\bm z & \bm \zbar}^\intercal$, it is block diagonal when stacking $\bmat{z_1 & \zbar_1 & z_2 & \zbar_2 & \dots z_n & \zbar_n}^\intercal$.

We also introduce the matrix $\bm P_e$ that selects the states related to the edge $e$ from our stacked vector:
\begin{align}
    \bm P_{e} \bmat{\bm z \\
\bm \zbar} = \bm P_{(n, m)} \bmat{\bm z \\
\bm \zbar} = \bmat{z_n \\ \zbar_n \\ z_m \\ \zbar_m}\, .
\label{eq:P-e on z zbar}
\end{align}
The $\bm P_{e}$ are isometries, but $\bm P_e^\dagger \bm P_{e}$ are not mutually orthogonal. Therefore, a matrix built from $4\times4$ matrices $\bm A_e$ as
\begin{align}
    \sum_e \bm P_e^\dagger \bm A_e \bm P_{e}\, ,
\end{align}
is not block diagonal. However, it can be written as the projection of a block diagonal matrix $\bigoplus_e \bm A_e$ and we write:
\begin{align}
    \sum_e \bm P_e^\dagger \bm A_e \bm P_{e} = \bm {B}_+^\dagger \bigoplus_e \bm A_e \bm{B_+}\, ,
\end{align}
for an according $4 E \times 2N $ matrix $\bm {B_+}$ that fulfills this equation.

\section{Phase stability preliminaries}
\label{sec:phase stability preliminaries}
Our results are based on the Generalized Small Phase Theorem of Chen \textit{et al.} \cite{chen_phase_2024}. We prove a straightforward proposition stating that if the transfer matrices of the system under consideration have a block structure, the global stability conditions can be decomposed into local conditions. An immediate application are networked systems that consist of node and edge variables that are coupled according to a graph.

Using this proposition we give a precise statement of the stability conditions for a power grid of general grid-forming grid actors with $V$-$q$ droop as introduced above.

For completeness, we begin by recalling the Small Phase Theorem of \cite{chen_phase_2024}, which provides conditions for the stability of the connected system ${\bm G} \# {\bm H}$, in terms of the \emph{numerical range} $W$ and the \emph{angular field of values} $W'$ \cite[Sec.~1.0, Def.~1.1.2]{horn_topics_1991}, \cite{wang_phases_2020,wang_phases_2023}, defined for a matrix ${\bm M}\in\C^{N\times N}$ as
\begin{align}
\label{eq:numerical range}
 W({\bm M}) &= \left\{{\bm z}^\dagger {\bm M}{\bm z} ~|~ {\bm z}\in\C^N\, ,~ {\bm z}^\dagger {\bm z} = 1\right\}\, , \\
 W'({\bm M}) &= \left\{{\bm z}^\dagger {\bm M}{\bm z} ~|~ {\bm z}\in\C^N\, ,~ {\bm z}^\dagger {\bm z}> 0\right\}\, .
\end{align}
When the numerical range lies in a half complex plane, we introduce the notion of \emph{sectoriality}. 
Assume that 0 is not in the interior of $W(\bm M)$. Define $\overline{\phi}({\bm M})$ and $\underline{\phi}({\bm M})$ as the maximum and minimum arguments of the elements of such a $W({\bm M})$, and $\delta({\bm M}) \coloneqq \overline{\phi}({\bm M}) - \underline{\phi}({\bm M})$. Then the matrix $\bm M$ is
\begin{itemize}
    \item \emph{semi-sectorial} if $\delta({\bm M}) \leq \pi$;
    \item \emph{quasi-sectorial} if $\delta({\bm M}) < \pi$;
    \item \emph{sectorial} if $0 \notin W(M)$.
\end{itemize}
Notice that a non-sectorial matrix ${\bm M}$ is semi-sectorial if $0$ is on the boundary of $W({\bm M})$.

Let ${\cal RH}_\infty^{m\times m}$ denote the set of $m\times m$ transfer matrices of real-rational proper stable systems. For these systems, all the poles of any ${\bm H}(s)\in \mathcal{RH}_\infty^{m\times m}$ (should there be any) are in the open left-hand side of the plane.
A system $\bm G\in {\cal RH}_\infty^{m\times m}$ is called frequency-wise sectorial if $\bm G(s)$ is sectorial for all $s\in j\R$. A system $\bm G(s)$ is semi-stable if its poles are in the closed left half plane. Take $j \Omega$ the set of poles on the imaginary axis, and $j \R \setminus j \Omega$ the indented imaginary axis with half-circles of radius $\epsilon \in\mathbb{R}$ around the poles and of radius $1/\epsilon$ around $\infty$ if it is a zero. These $\epsilon$-detours lie in the right half-plane. We call this indented imaginary axis ``the contour''. A system is semi-stable frequency-wise semi-sectorial if $\bm G(s)$ has constant rank along the contour and is semi-sectorial on $j \R \setminus j \Omega$.

The phase center is defined as $\gamma[\bm G(s)] \coloneqq \left\{ \overline{\phi}[\bm G(s)] + \underline{\phi}[\bm G(s)] \right\}/2$, and without loss of generality, we assume that $\gamma[\bm G(\epsilon^+)] \coloneqq \lim_{\epsilon\searrow 0}\gamma[\bm G(\epsilon)] = 0$.

We can now recall Chen {\it et al.}'s Small Phase Theorem.

\begin{theorem}[Generalized Small Phase Theorem, \cite{chen_phase_2024}]\label{thm:small-phase}
    Let ${\bm G}$ be semi-stable frequency-wise semi-sectorial with $j \Omega$ being the set of poles on the imaginary axis, and $\bm H\in{\cal RH}_\infty$ be frequency-wise sectorial. 
    Then ${\bm G}\# {\bm H}$ is stable if 
    \begin{align}
    \label{eq:small phase condition max phases}
        \sup_{s\in j[0,\infty]\setminus j\Omega}\left[\overline{\phi}({\bm G}(s)) + \overline{\phi}({\bm H}(s))\right] &< \pi\, , \\ 
        \label{eq:small phase condition min phases}
        \inf_{s\in j[0,\infty]\setminus j\Omega}\left[\underline{\phi}({\bm G}(s)) + \underline{\phi}({\bm H}(s))\right] &> -\pi\, .
    \end{align}
\end{theorem}

\begin{proof}
    See \cite{chen_phase_2024}
\end{proof}

If the system $\bm G \#\bm H$ has a block structure, e.g., a networked distributed power system, we can show the following:

\begin{pro}[Generalized Small Phase Theorem with Block Structure]
\label{thm:small_phase_theorem_block}
    Consider the system $\bm G\# \bm H$ with the block structure $\bm H =  \bigoplus_{n}\Tn_n(s) $ and $\bm G = {\bm B}^\dagger\bigoplus_{e}\Te_e(s){\bm B}$ for some $\bm B$ of appropriate dimensions.
    For each $n$, let $\Tn_n(s)\in{\cal RH}_\infty$ be frequency-wise sectorial. 
    For each $e$, let $\Te_e(s)$ be semi-stable frequency-wise semi-sectorial individually and along the indented imaginary axis avoiding the poles of all $\Te_e (s)$ for indents smaller than some finite $\epsilon^*$. Write $j \Omega$ for the union of the set of poles on the imaginary axis. 
    Assume that $\bm G(s)$ has constant rank along the contour.
    Then, the interconnected system $\bm G \# \bm H$ is stable if 
    \begin{align}
    \label{eq:small phase theorem, block structure version, sectoriality condition H}
        \max_n\overline{\phi}\left(\Tn_n(s)\right) - \min_n\underline{\phi}\left(\Tn_n(s)\right) &< \pi\, , 
\end{align}
for all $s\in j[0,\infty]$, and
\begin{align}
        \max_e\overline{\phi}\left(\Te_e(s)\right) - \min_e\underline{\phi}\left(\Te_e(s)\right) &\leq \pi\, ,
        \label{eq:small phase theorem, block structure version, semi-sectoriality condition G}
    \end{align}
    for all $s\notin j\Omega$, and 
    \begin{align}
    \label{eq:small phase theorem, block structure version, stability condition sup}
        \sup_{n,e,s\notin j\Omega}\left[\overline{\phi}\left(\Tn_n(s)\right) + \overline{\phi}\left(\Te_e(s)\right)\right] &< \pi\, , \\ 
        \inf_{n,e,s\notin j\Omega}\left[\underline{\phi}\left(\Tn_n(s)\right) + \underline{\phi}\left(\Te_e(s)\right)\right] &> -\pi\, .
        \label{eq:small phase theorem, block structure version, stability condition inf}
    \end{align}
\end{pro}

\textit{Remark:}
$\bm H$ is stable, and its sectoriality is ensured by \eqref{eq:small phase theorem, block structure version, sectoriality condition H}. $\bm G$ is semi-stable, and its semi-sectoriality is ensured by \eqref{eq:small phase theorem, block structure version, semi-sectoriality condition G} and the rank condition.
Equations~\eqref{eq:small phase theorem, block structure version, stability condition sup}-\eqref{eq:small phase theorem, block structure version, stability condition inf} imply the stability condition of Theorem~\ref{thm:small-phase}.

\begin{proof}
    We provide the proof in Appendix \ref{sec:proof_of_phase_theorem, block version}.
\end{proof}

\section{Linear form of power grids with \textnormal{$V$-$q$} droop}
\label{app:linear form of power grids with v-q droop}

To make use of Proposition \ref{thm:small_phase_theorem_block} we have to linearize the power grid model under investigation into an appropriate form.
In this section, we show that the power grid can be represented as an interconnected feedback system of two transfer matrices: $\Tnod \# \Tnet$.
$\Tnod$ includes all nodal transfer matrices from $\hat q_n$ and $p_n$ to $\varrho_n$ and $\omega_n$, as in \eqref{eq:definition_entries_T_n}.
$\Tnet$ represents the network structure and the physics of the coupling, as it takes $\bm \varrho$ and $\bm\omega$ as inputs and provides $\bm{\hat q}$ and $\bm p$ as outputs.
The fundamental assumption we make is that the nodes can be modeled as voltage sources that react to conditions in the grid. This assumption is most natural in the context of grid-forming actors, such as power plants or grid-forming inverters.

\subsection{Complex frequency notation}

As noted above, every node has a complex voltage (representing a balanced three-phase voltage) $v_n = v_{d,n} + j v_{q,n}$:
\begin{align}
    v_n(t) &= V_n(t) e^{j \varphi_n(t)} = e^{\theta_n(t)} \, ,
    \label{eq:voltage_phasor}
\end{align}

and a complex current $\imath_n$. The latter is given in terms of the former through the admittance matrix $\bm Y$:
\begin{align}
\bm \imath(t) = \bm Y \cdot \bm v(t) = - j \bm L \cdot \bm v(t)\, .
\end{align}

The matrix $\bm L := je^{-j\kappa} \bm Y \in \mathbb{R}^{N\times N}$ is a real, symmetric, positive definite Laplacian.
We show the proof for lossless grids, where $\kappa = 0$. The lossy case goes analogously with a rotation, see Section~\ref{sec:lossy lines}.

We use a power-invariant transformation from $ABC$ coordinates, so that the apparent power is given by $S_n(t) = v_n(t) \ibar_n(t) = p_n(t) + j q_n(t)$ with active power $p_n(t)$ and reactive power $q_n(t)$.

Milano \cite{milano_complex_2022} suggests writing the nodal dynamics through the time derivative of the complex phase $\theta_n$, the complex frequency $\eta$:
\begin{align}
	\eta_n(t) &= \dot \theta_n(t)\, , \\
    \dot v_n(t) &= \eta_n(t) v_n(t)\\
    &= (\varrho_n(t) + j \omega_n(t)) v_n(t)\, .
\end{align}

We will drop the explicit time dependence $(t)$ from now on. By considering both, the complex equation and the complex conjugate equation,
\begin{align}
    \dot v_n &= \eta_n v_n\, ,\\
    \dot \vbar_n &= \etabar_n \vbar_n\, ,
\end{align}

we can switch back and forth between complex and real picture, using a linear transformation. The velocities $\varrho_n$, $\omega_n$, $\eta_n$ and $\etabar_n$ are related by:
\begin{align}
\bmat{\eta_n \\ \etabar_n} &= \bmat{1 & j \\ 1 & - j} \bmat{\varrho_n \\ \omega_n} = \bm U \bmat{\varrho_n \\ \omega_n}\, , \\
\bmat{\varrho_n \\ \omega_n} &= \frac12 \bmat{1 & 1 \\ -j & j} \bmat{\eta_n \\ \etabar_n} = \frac12 \bm U^\dagger \bmat{\eta_n \\ \etabar_n}\, ,
\end{align}

Note that $\bm U^{-1} = \frac{1}{2} \bm U^\dagger$, thus $\bm U/\sqrt{2}$ is a unitary matrix. This means that under $\bm U$ as coordinate transformation, all pertinent properties of linear dynamical systems are retained.

\subsection{A system of grid-forming actors}\label{sec:grid-forming actors}

We are interested in conditions that guarantee small-signal stability of a heterogeneous system of grid-forming actors, without strong assumptions on their internal structure. As noted above, we assume that we can model the nodes as voltages reacting to the grid state.  We assume that the voltages react in a smooth, differentiable manner, and that $V_n > 0$. Thus, $\omega_n$ and $\varrho_n$ are defined, and can be chosen as the nodal output variable. Using $p_n$ and $q_n$ as the input that the nodal actor sees from the grid, we can write the general form of a node's behavior in terms of three functions $r_n$, $o_n$ and ${\bm f}^{\bm x}_n$:
\begin{align}
     \label{eq:node no assumptions 1}
    \varrho_n &= r_n(\varphi_n, V_n, p_n, q_n, \bm x_n)\, , \\
    \omega_n &= o_n(\varphi_n, V_n, p_n, q_n, \bm x_n)\, , \\
	\dot {\bm x}_n &= {\bm f}^{\bm x}_n(\varphi_n, V_n, p_n, q_n, \bm x_n)\, .
  \label{eq:node no assumptions 3}
\end{align}
Here, $\bm x_n\in \mathbb{R}^{n_\text{var}}$ are internal states of dimension $n_\text{var}$ that reflect the inner workings of the grid actor, and are not visible directly in the output $v$. Examples include generator frequencies, inner-loop DC voltages, or the $d$- and $q$-components of internal AC quantities.

We make two assumptions on the form of the functions $r_n$, $o_n$ and ${\bm f}^{\bm x}_n$: I) Following \cite{kogler_normal_2022}, we assume that the nodal dynamics does not explicitly depend on $\varphi_n$. This assumption is %usually true, and is
justified by symmetry considerations and the desire to not introduce harmonic disturbances into the grid. II) We assume that the reaction to a deviation in the voltage mirrors that of a deviation in the reactive power. That is, we assume that near the operation point, $r_n$, $o_n$ and ${\bm f}^{\bm x}_n$ only depend on $\hat q_n = q_n + \alpha_n V_n$ for some real $\alpha_n$ rather than on both $q_n$ and $V_n$ separately. With these assumptions we have:
\begin{align}
     \label{eq:total system in normal form 1}
    \varrho_n &= r_n(p_n, \hat q_n, \bm x_n)\, , \\
    \omega_n &= o_n(p_n, \hat q_n, \bm x_n)\, , \\
	\dot {\bm x}_n &= {\bm f}^{\bm x}_n(p_n, \hat q_n, \bm x_n)\, .
 \label{eq:total system in normal form 3}
\end{align}

\subsection{The linearized nodal response}

We define the coefficients of the Jacobian as 
\begin{align}
J_n^{\omega p} &\coloneqq \frac{\partial o_n}{\partial p_n}\, , &
J_n^{\varrho \hat q} &= \frac{\partial r_n}{\partial (\hat q_n)}\, , &
\bm J_n^{x x} &= \frac{\partial \bm f_n^x}{\partial \bm x_n}\, , & \text{etc.}
\end{align}

We now want to look at the linear response of the nodal subsystem around an operating point $v_n^\circ$ , $i_n^\circ$. We assume that the operating point satisfies $\varrho_n^\circ = \omega_n^\circ = \dot {\bm x}_n = 0$. Write $\Delta p_n = p_n - p_n^\circ$ and $\Delta \hat q_n = q_n - q_n^\circ + \alpha_n(V_n - V_n^\circ)$ and assume that ${\bm x}_n^\circ = \bm 0$.
The linearized nodal dynamics are then
\begin{align}
    \dot {\bm x}_n &= \bm J_n^{x p} \Delta p_n + \bm J_n^{x q} \Delta \hat q_n + \bm J_n^{x x} \bm x_n\, , \\
    \varrho_n &= J_n^{\varrho p} \Delta p_n + J_n^{\varrho \hat q} \Delta \hat q_n + \bm J_n^{\varrho x} \bm x_n\, , \\
    \omega_n &= J_n^{\omega p} \Delta p_n + J_n^{\omega \hat q} \Delta \hat q_n + \bm J_n^{\omega x} \bm x_n\, .
\end{align}

which we stack as 
\begin{align}
    \dot {\bm x}_n &= \bm J_n^{x q p} \bmat{\Delta \hat q_n \\ \Delta p_n}  + \bm J_n^{x x} \bm x_n\, , \\
    \bmat{\varrho_n \\ \omega_n} &= \bm J_n^{\varrho \omega \hat q p}\bmat{\Delta \hat q_n \\ \Delta p_n} + \bm J_n^{\varrho \omega x} \bm x_n\, .
\end{align}

The nodal transfer matrix from $\bmat{\Delta \hat q_n & \Delta p_n}^\intercal$ to $\bmat{\varrho_n & \omega_n}^\intercal$ is then just
\begin{align}
    -\Tn_n(s) &= \bm J_n^{\varrho \omega \hat q p} + \bm J_n^{\varrho \omega x} (s - \bm J_n^{x x})^{-1} \bm J_n^{x q p}\, .
    \label{eq:T_n}
\end{align}

We can summarize the transfer matrices of all nodes in $\Tnod$ such that
\begin{align}\label{eq:nodal-response}
   \bmat{\bm \varrho \\ \bm \omega} &= \Tnod \bmat{\Delta \bm{\hat q}\\ \Delta \bm p} \coloneqq \bigoplus_n \Tn_n(s) \bmat{\Delta \bm{\hat q} \\ \Delta\bm p}.
\end{align}

\subsection{The linearized network response}

To obtain the full linearized equations, we need the response of $\Delta p_n$ and $\Delta \hat q_n$ to variations in the complex angle $\theta_n$ around a given power flow with $\theta_n^\circ$.

This is most easily given in terms of a variant of the complex power and the complex couplings introduced by \cite{buttner_complex_2024}. We define
\begin{align}
\sigma_n &\coloneqq q_n + j p_n\, ,
\end{align}
to mirror the definition of the complex frequency \cite{milano_complex_2022}. In terms of the usual complex power, this is $\sigma_n = j\overline{S}_n$. This complex power can be expressed in terms of the Hermitian matrix $\bm\Kappa\in\mathbb{C}^{N\times N}$ of complex couplings \cite{milano_complex_2022,buttner_complex_2024}:
\begin{align}
\Kappa_{nm} &= \vbar_n L_{nm} v_m\, , \\
\sigma_n &= \sum_m \Kappa_{nm}\, . \label{eq:sum K is sigma}
\end{align}

These quantities have a very simple derivative with respect to the complex phases of the system:
\begin{align}
\frac{\partial \Kappa_{nm}}{\partial \theta_h} &= \delta_{hm} \Kappa_{nm}\, ,\quad
&\frac{\partial \Kappa_{nm}}{ \partial\thetabar_h} &= \delta_{hn}  \Kappa_{nm}\, ,\\
\frac{\partial \sigma_n}{\partial \theta_h} &= \Kappa_{nh}\, ,
&\frac{\partial \sigma_n}{\partial \thetabar_h} &= \delta_{nh} \sigma_n\, . 
\end{align}

The linearization of $\sigma_n$ around an operating state of the system with complex couplings $\Kappa^\circ_{nm}$ and complex power $\sigma^\circ_n$ is then given by
\begin{align}
\sigma_n \approx \sigma^\circ_n + \sigma_n^\circ \Delta \thetabar_n + \sum_m \Kappa^\circ_{nm} \Delta \theta_m\, 
\end{align}

or, in vector notation, 
\begin{align}
\bmat{\Delta \bm \sigma \\ \Delta \bm \sigmabar } \approx \bmat{\bm \Kappa^\circ & [\bm \sigma^\circ] \\ [\bm \sigmabar^\circ] & \bm \Kappabar^\circ} \bmat{\Delta \bm \theta \\ \Delta \bm \thetabar}\, .
\end{align}

As the nodal dynamics depend on $\Delta \hat q_n $ and $\Delta p_n$, as inputs, we now consider 
\begin{align}
\Delta \sigma_n + \alpha_n \Delta V_n = \Delta \hat q_{n} + j \Delta p_n\, ,
\end{align}

for the output of the edge dynamics.
Together with $\Delta V_n \approx V_n^\circ \frac12 ( \Delta \theta + \Delta \thetabar)$, we obtain
\begin{align}
&\bmat{\Delta \bm \sigma + \bm \alpha \Delta \bm V \\ \Delta \bm \sigmabar + \bm \alpha \Delta \bm V}
\approx \bm J^\text{net} \bmat{\Delta \bm \theta \\ \Delta \bm \thetabar}\, ,
\end{align}
with the transfer matrix
\begin{align}
\bm J^\text{net} \coloneqq&
\bmat{\bm \Kappa^\circ + \frac12 [\bm \alpha ][\bm V^\circ] & [\bm \sigma^\circ] + \frac12 [\bm \alpha ][\bm V^\circ] \\ [\bm \sigmabar^\circ] + \frac12 [\bm \alpha ][\bm V^\circ] & \bm \Kappabar^\circ + \frac12 [\bm \alpha ][\bm V^\circ] }\, .
\end{align}

Note that as $\bm K^\circ$ is Hermitian, and so is $\bm J^\text{net}$. Further, we see from \eqref{eq:sum K is sigma} that $\bmat{\bm 1 & - \bm 1}^\intercal$ is a zero mode of the network response $\bm J^\text{net}$.

At this point, we can see the necessity of incorporating the $V$-$q$ droop into the network response. Without the presence of the $\alpha_n$, $\bm J^\text{net}$ would be indefinite and thus not amenable to sectorial analysis.

\subsection{The full system}

Above we derived the nodal transfer matrix from $p_n$, $q_n + \alpha_n V_n$ to $\varrho_n$ and $\omega_n$, and the network response from $\theta_n$ and $\thetabar_n$ to $\sigma_n + \alpha_n V_n$ and $\sigmabar_n + \alpha_n V_n$. We can now combine these into the full system equations.
Recall that 
\begin{align}
\Delta \dot \theta_n &= \eta_n\, , \\
s \Delta \theta_n &= \eta_n\, ,
\end{align}
where the latter equation is in Laplace space.
Let us introduce $\bm{\tilde U} \in\mathbb{C}^{2N\times2N}$,
\begin{align}
    \bm {\tilde U} &= \bigoplus_n \bm U\, ,
\end{align}

With this we can write the network response from a deviation in $\varrho$ and $\omega$ to a deviation in $\hat q$ and $p$ as
\begin{align}
\Tnet(s) &= \frac{1}{2} \bm {\tilde U}^{\dagger} \frac{1}{s} \bm J^\text{net} \bm {\tilde U}\, .
\label{eq:relation_Jlin_Tnet}
\end{align}

The major remaining challenge to applying Proposition \ref{thm:small_phase_theorem_block} and getting decentralized conditions, is to decompose this matrix into edge-wise contributions. As we will see in the next section, we can treat the network response as a superposition of two-node systems.

The full system $\Tnod \# \Tnet$ then has the structure
\begin{align}
\bigoplus_n \Tn_n(s) \; \# \; \Tnet(s)\, .
\label{eq:full-system}
\end{align}

\section{Proof of the main Proposition 
\ref{pro:main:power_grid}}

We now proceed to the proof of the main proposition.
\label{sec:Proof: Linear stability of power grids with $V$-$q$ droop}
The first step is to provide conditions for the sectoriality of the nodal transfer matrices. Then we provide the edge-wise decomposition of the network response, and demonstrate under which conditions it is semi-stable frequency-wise semi-sectorial. The main Theorem then follows by applying Proposition~\ref{thm:small_phase_theorem_block}.

\subsection{Sectoriality of the nodal transfer matrix} % $\Tnod$
Each $\Tn_n(s)$ of the form \eqref{eq:definition_entries_T_n} is a complex $2\times2$ matrix. Here, we give conditions that ensure that it is strictly accretive, meaning the numerical range is contained in the open right half plane: $\underline \phi > - \pi/2$ and $\overline \phi < \pi/2$.
Is gives especially concise conditions for sectoriality.

\begin{lem}
\label{lem:accretive T_n}
    A complex $2\times 2$ matrix $\Tn_n(s)$ is strictly accretive, hence sectorial, if and only if its four entries [see \eqref{eq:definition_entries_T_n}] fulfill \eqref{eq:tr_T_nodes_definite} and \eqref{eq:det_T_nodes_positive}:
    \begin{align}
    \Re({T}^{\omega p}_n) + \Re({T}^{\varrho \hat q}_n) &> 0\, , \\
    \Re({T}^{\omega p}_n) \cdot\Re({T}^{\varrho \hat q}_n) &>\frac{1}{4} \left|{T}^{\omega \hat q}_n + \overline {T}^{\varrho p}_n\right|^2\, .
\end{align}
\end{lem}

\begin{proof}
If the numerical range $W$ of $\Tn_n$ (see \eqref{eq:numerical range}) is contained in the right-hand side, the real part of $W(\Tn_n)$ has to be strictly positive: $\Re(W(\Tn_n(s))) > 0$. The real part of the numerical range is given by the numerical range of the Hermitian part of $\Tn_n(s)$, which we denote $\bm {\hat T}_n(s) = \frac{1}{2} (\Tn_n(s) + \Tn_n(s)^\dagger)$.
The numerical range of a Hermitian matrix is on the real axis. It is strictly positive if and only if the matrix is positive definite. The two by two matrix $\bm {\hat T}_n(s)$ is positive definite if and only if its determinant and its trace are positive. Expressed in terms of the matrix elements of $\Tn_n(s)$ these conditions are \eqref{eq:tr_T_nodes_definite} and \eqref{eq:det_T_nodes_positive}.
\end{proof}

\subsection{Edge-wise decomposition and analysis of the network response} % $\Tnet$
We now return to the network response. Our goal is to show that under the condition that [see \eqref{eq:alpha_bound_T_lines_definite}]
\begin{align}
    \alpha_n \ge \alpha^{\text{theory}}_n \coloneqq 2 \sum_{m} \tilde Y_{nm}\frac{V_m^\circ}{\cos(\varphi_n^\circ - \varphi_m^\circ)},
\end{align}
we can decompose the network response into frequency wise semi-stable and semi-sectorial edge contributions.

\begin{lem}
   \textnormal{$\bm J^\text{net}$} can be decomposed into edge-wise contributions $\bm J_e$ such that
\begin{align}
\label{eq: edge-wise decomposition of J_net}
    \bm J^\textnormal{net} = \bm{B_+}^\dagger \bigoplus_e \bm J_e \bm{B_+}\, ,
\end{align}
if we introduce an edge-wise decomposition $\alpha'_{nm}$ of $\alpha_n$ such that
\begin{align}
\label{eq:alpha_n edge-wise decomposition}
    \alpha_n = - 2V^\circ_n \sum_{m \neq n} L_{nm} \alpha'_{nm}.
\end{align}
\end{lem}

\begin{proof}
The fundamental strategy is to collect the terms that represent each edge. 
In each of the four blocks of $\bm J^\text{net}$, the off diagonal matrix elements naturally have an edge associated to them. The diagonal elements of $\bm K^\circ$ can be written as a sum of edge-wise contributions $K^\circ_{nn} = - |V^\circ_n|^2 \sum_{m \neq n} L_{nm}$. The $\bm \sigma^\circ$ can be written as $\sigma^\circ_n = \sum_{m \neq n} K^\circ_{nm} - |V^\circ_n|^2 \sum_{m \neq n} L_{nm}$. We then introduce a similar decomposition for $\frac12 \bm \alpha$ times $\bm V^\circ$, writing $\frac 12 \alpha_n V_n^\circ = - |V^\circ_n|^2 \sum_{m \neq n} L_{nm} \alpha'_{nm}$. Now, the contributions to the matrix elements of $\bm J^\text{net}$ associated to an edge $e = (n,m)$ all live on the rows and columns associated to $n$ and $m$. Thus, we can place them in %the full matrix
a $4\times 4$ matrix $\bm J^e$
using the matrices $\bm P_e$ of \eqref{eq:P-e on z zbar} that pick out exactly those rows and columns.

To collect these edge-wise contributions, we introduce %$\beta_{nm} = - |V^\circ_n|^2 L_{nm} (1 + \alpha'_{nm})$ and
$C'_{nm} \coloneqq \frac{\vbar_n^\circ}{v_n^\circ} (1 + \alpha'_{nm}) - \frac{\vbar_m^\circ}{v_n^\circ}$. Then we can succinctly write the four by four matrix of elements originating from a single edge as $\bm J_e = -L_{nm} \bm R^\dagger \bm{\tilde{J}}_e \bm R$ with
\begin{align}
    \bm {\tilde J}_e &= \bmat{ 1 + \alpha'_{nm} & C'_{nm} & - 1 & 0
        \\
        \Cbar'_{nm} & 1 + \alpha'_{nm} & 0 & - 1
    \\
        -1 & 0 & 1 + \alpha'_{mn} & C'_{mn}
        \\
        0 & -1 & \Cbar'_{mn} & 1 + \alpha'_{mn}}.
\end{align}
and $\bm R := \text{diag}(\vbar_n^\circ, v_n^\circ, \vbar_m^\circ, v_m^\circ)$.
With this \eqref{eq: edge-wise decomposition of J_net} can be verified by straightforward calculation, collecting all terms associated to each edge.
\end{proof}

As $\bm J^\text{net}$, and the $\bm J_e$, are Hermitian, their numerical range is on the real axis. They are (semi-)sectorial, if and only if they are (semi-)definite. In the phase stability theorems, it is assumed that the transfer matrix $G(\epsilon^+)$ has phase center zero. From \eqref{eq:relation_Jlin_Tnet} we see that this implies that $\bm J^\text{net}$ and thus $\bm J_e$ have to be positive semi-definite.

\begin{lem}
\label{lem:edge-wise alpha bounds and phase differences}
$\bm J_e$ is positive semi-definite, hence semi-sectorial, if 
\begin{gather}
\label{eq:phase_bounds}
|\varphi_n^\circ - \varphi_m^\circ| < \frac{\pi}{2}\,  
\quad\forall\, e=(n,m)\in \mathcal{E},\\
\label{eq:edge-wise_alpha_bound}
    \alpha'_{nm} \ge \frac{V^\circ_m}{V^\circ_n \cos(\varphi_n^\circ - \varphi_m^\circ)} - 1\, .
\end{gather}
\end{lem}

\begin{proof}
This can be verified with a straightforward calculation, e.g.,  using the Schur complement lemma.
\end{proof}

The edge-wise decomposition of $\alpha_n$ leaves us with the freedom to weight the $\alpha'_{nm}$ freely, as long as they satisfy \eqref{eq:alpha_n edge-wise decomposition}.
The tightest bound is achieved by weighting them proportional to the bounds derived in \eqref{eq:edge-wise_alpha_bound}. However, we can achieve a much more concise node-wise condition for the $\alpha_n$, which are actual dynamical parameters of the nodal actors.

\begin{lem}
\textnormal{$\Tnet(s)$} can be decomposed into semi-stable frequency-wise sectorial $\Te_e$ as
\begin{align}
    \Tnet(s) = \bm{\tilde U}^\dagger \bm{B_+}^\dagger \bigoplus_e \Te_e(s) \bm{B_+} \bm{\tilde U}\, ,
\end{align}
if  $\alpha_n \ge \alpha_n^\textnormal{theory}$, i.e., \eqref{eq:alpha_bound_T_lines_definite} holds.
\end{lem}

\begin{proof}
    The $\Te_e(s)$ are given by
\begin{align}\label{eq:relation_Je_Te}
    \Te_e \coloneqq \frac{1}{2s} \bm J_e\, .
\end{align}
According to Lemma~\ref{lem:edge-wise alpha bounds and phase differences}, \eqref{eq:phase_bounds} and \eqref{eq:edge-wise_alpha_bound} imply frequency-wise semi-sectorial $\bm J_e$ and thus $\Te_e$.
The factor $1/s$ makes them semi-stable, because the pole is at zero and the rank left constant along the contour.
Using the definition of $\alpha'_{nm}$ we see that \eqref{eq:edge-wise_alpha_bound} can always be satisfied if $\alpha_n\ge\alpha_n^\text{theory}$.
\end{proof}

As $\bigoplus_e \Te_e(s)$ only depends on $s$ through scaling by a common factor, we also immediately have that its rank is constant along the contour.
Thus, $\Tnet(s)$ is semi-stable frequency-wise semi-sectorial.
On $ s \in j(\epsilon^+, \infty]$ the phases of the $\Te_e(s)$ are simply:
$    \underline{\phi}(\Te_e) = -\frac{\pi}{2}=     \overline{\phi}(\Te_e)  .$
on the quarter circle of radius $\epsilon^+$ from $j \epsilon^+$ to $\epsilon^+$, they rotate to $0$.

In conclusion, \eqref{eq:alpha_bound_T_lines_definite} ensures semi-stable frequency-wise sectorial $\Tnet(s)$ with a DC phase center of 0, which is a pole, and all phases $-\frac{\pi}{2}$ at $s \in j(\mathbb{R}\setminus\Omega)$.

\subsection{Putting everything together}

\begin{proof} We can now apply Proposition~\ref{thm:small_phase_theorem_block} to the system given by \eqref{eq:full-system}, with $\bm H = \Tnod = \bigoplus_n \Tn_n(s)$, $\bm B =\bm{B_+} \bm{\tilde U}$, and  $\bm G = \Tnet = \bm B^\dagger \bigoplus_e \Te_e(s) \bm B$.
We have shown in the previous sections that with \eqref{eq:tr_T_nodes_definite}-\eqref{eq:alpha_bound_T_lines_definite}, (i) the $\Tn_n$ in \eqref{eq:definition_entries_T_n} and \eqref{eq:T_n} are in $\mathcal{RH}_\infty$ (hence stable) and frequency-wise sectorial according to Lemma \ref{lem:accretive T_n};
(ii) the $\Te_e$ in 
\eqref{eq:relation_Je_Te}
are semi-stable frequency-wise semi-sectorial according to Lemma \ref{lem:edge-wise alpha bounds and phase differences}.
They are also semi-stable along the shared indented imaginary axis because they share the same poles.
Finally, $\bm G$ has constant rank along the contour, because it depends on $s$ only by a prefactor $1/s$.

We now proceed to show that \eqref{eq:small phase theorem, block structure version, sectoriality condition H}-\eqref{eq:small phase theorem, block structure version, stability condition inf} hold.
Equation \eqref{eq:small phase theorem, block structure version, sectoriality condition H} is fulfilled for \eqref{eq:tr_T_nodes_definite}-\eqref{eq:det_T_nodes_positive}, as $\underline\phi(\Tn_n) > - \pi/2$ and $\overline\phi(\Tn_n) <  \pi/2$.
Equation \eqref{eq:small phase theorem, block structure version, semi-sectoriality condition G} is fulfilled, as
$    \underline{\phi}(\Te_e) = -\frac{\pi}{2}=     \overline{\phi}(\Te_e)  .$
Similarly, the combined phases of $\Tnet$ and $\Tnod$ lie within $(-\pi,0)$ at all $s \in j(\mathbb{R}\setminus\Omega)$, hence
\eqref{eq:small phase theorem, block structure version, stability condition sup} and \eqref{eq:small phase theorem, block structure version, stability condition inf} hold. This concludes the proof.
\end{proof}

\begin{figure}
    \centering
    \includegraphics[width=\columnwidth]{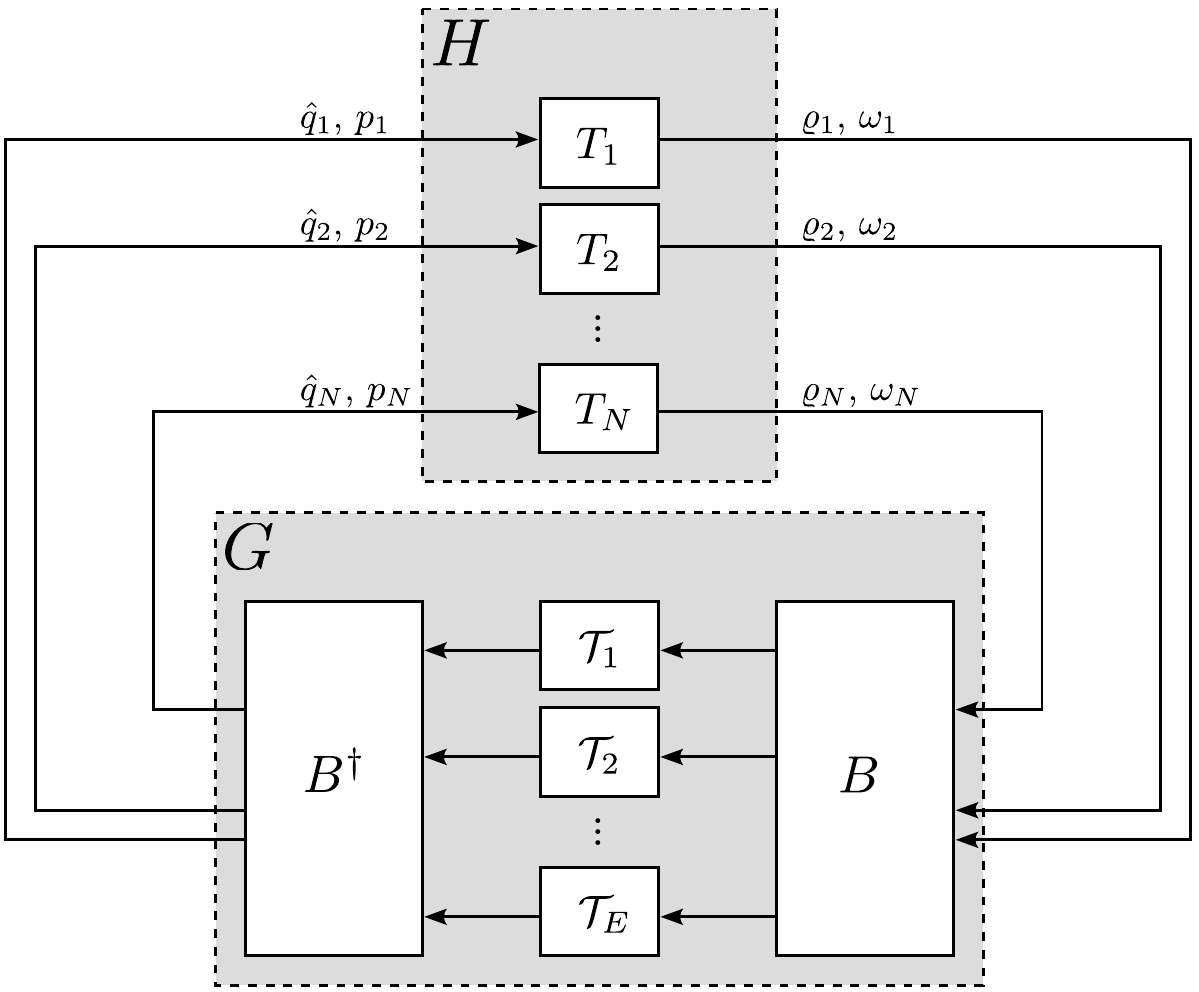}
    \caption{Block diagram representation of the system considered. 
    Block $H$ is the nodal response to the lines' output, and block $G$ is the lines' response to the nodes' dynamics.}
    \label{fig:block-diagram}
\end{figure}

As $\Te_e$ have phase $-\frac{\pi}{2}$ at all non-zero frequencies, the phases of $\Tn_n$ need not be contained in the open right half plane. However, $\bm{\hat T}_n > 0$ is sufficient for our examples below and gives the most concise conditions.

\section{Proof of Proposition \ref{thm:small_phase_theorem_block}}
\label{sec:proof_of_phase_theorem, block version}

\subsection{Preliminaries}

Let us recall two properties of $W'$ that will prove useful later on. 
First, it follows from the definition of $W'$ that 
\begin{align}
    W'({\bm B}^\dagger{\bm M}{\bm B}) &\subseteq\left( W'({\bm M}) \cup 0\right)\, ,
    \label{eq:sub-field}
\end{align}
for any $\bm M \in \mathbb{C}^{m\times m}$ and ${\bm B}$ of appropriate size, and therefore, 
\begin{align}
    \overline{\phi}({\bm B}^\dagger{\bm M}{\bm B}) &\leq \overline{\phi}({\bm M})\, , & 
    \underline{\phi}({\bm B}^\dagger{\bm M}{\bm B}) &\geq \underline{\phi}({\bm M})\, .
\end{align}
Second, for a block diagonal system ${\bm M} = \bigoplus_e{\bm M}_e$, the numerical range is the convex hull of the blocks' numerical ranges \cite[Property 1.2.10]{horn_topics_1991}:
\begin{align}
    W({\bm M}) &= {\rm Conv}\left(W({\bm M}_1),...,W({\bm M}_E)\right)\, .
    \label{eq:conv-hull}
\end{align}
Thus, if $\bm M$ is semi-sectorial, 
\begin{align}
    \overline{\phi}({\bm M}) &= \max_e\overline{\phi}({\bm M}_e)\, , & 
    \underline{\phi}({\bm M}) &= \min_e\underline{\phi}({\bm M}_e)\, .
    \label{eq:phase bounds block matrix}
\end{align}

With this toolbox, we are now ready to prove our main result. 
The proof of Proposition~\ref{thm:small_phase_theorem_block} relies on the four following Lemmas. 

\begin{lem}\label{lem:stability_block}
    Let ${\bm T}_1,...,{\bm T}_N$ be stable transfer functions. Then ${\bm T}(s) = \bigoplus_n{\bm T}_n(s)$ is stable.    
\end{lem}

\begin{proof}
    The transfer function ${\bm T}(s)$ is stable, because the set of its poles is the union of the poles of its blocks. 
\end{proof}

\begin{lem}
\label{lem:block_sectoriality}
    Let $\Tn_1,\ldots,\Tn_N$ be frequency-wise sectorial transfer functions. Then, $\bm T(s) = \bigoplus_n T_n(s)$ is frequency-wise sectorial if and only if 
    \begin{align}\label{eq:cond-fws}
        \max_n\overline{\phi}\left({\bm T}_n(s)\right) - \min_n\underline{\phi}\left({\bm T}_n(s)\right) &< \pi\, ,
    \end{align}
    for all $s\in j[0,\infty]$, cf. \eqref{eq:small phase theorem, block structure version, sectoriality condition H}.
\end{lem}
\begin{proof}

     Due to \eqref{eq:conv-hull}, we have that $W(\bm T)$ is the convex hull of all $W(\Tn_n)$. 
     Therefore, if \eqref{eq:cond-fws} is satisfied for all, $W(\bm T)$ is contained in a sector of angle $\delta(\bm T) < \pi$. 
     Furthermore, as none of the $W({\bm T}_n)$ contain the origin, $W(\bm T)$ does not contains the origin. 
     We conclude that ${\bm T}$ is frequency-wise sectorial. 
     Similarly, if ${\bm T}$ is frequency-wise sectorial, then none of the $W({\bm T}_n)$ contains the origin, and they all lie in a sector of angle smaller than $\pi$ and \eqref{eq:cond-fws} holds. 
     All of the above holds for any $s\in j[0,\infty]$, which concludes the proof. 
\end{proof}

\begin{lem}\label{lem:semi_stability_block}
    Let $\Te_1,...,\Te_E$ be semi-stable transfer functions and let us define $\bm{\mathcal T}(s) = \bigoplus_e\Te_e(s)$. 
    Let ${\bm B}$ be a complex matrix of appropriate dimensions. 
    Then both $\bm{\mathcal{T}}(s)$ and $\XTX{s}$ are semi-stable.
\end{lem}

\begin{proof}
    The transfer function $\bm{\mathcal{T}}(s)$ is semi-stable, because the set of its poles is the union of the poles of its blocks. 
    As the matrix ${\bm B}$ cannot introduce new poles, the poles of $\XTX{s}$ form a subset of the poles of ${\bm T}(s)$. 
    Therefore, $\XTX{s}$ is semi-stable.
\end{proof}

\begin{lem}\label{lem:semi_sectoriality_block}
    Let $\Te_1,...,\Te_E$ be frequency-wise semi-sectorial transfer functions and let us define ${\bm T}(s) = \bigoplus_e\Te_e(s)$. 
    Assume further that 
    \begin{align}
        \max_e\overline{\phi}(\Te_e(s)) - \min_e\underline{\phi}(\Te_e(s)) &\leq \pi\, ,
        \label{eq:assumption1}
    \end{align}
    for all $s \in j\R\setminus j\Omega$,
    where $j \Omega$ is the union of the poles of $\Te_1,...,\Te_E$ that lie on the imaginary axis, cf. \eqref{eq:small phase theorem, block structure version, semi-sectoriality condition G}.
    Assume that $\Te_1,...,\Te_E$ are all frequency-wise semi-sectorial, and assume furthermore that they are semi-sectorial along the indented imaginary axis avoiding the poles of all $\Te_e (s)$ for indents smaller than some finite $\epsilon^*$. 
    Finally, assume that $\XTX
{s}$ has constant rank along this indented imaginary axis for some constant complex matrix ${\bm B}$ of appropriate dimensions.
    Then $\XTX{s}$ is frequency-wise semi-sectorial. 
\end{lem}
\textit{Remark:}
$\bm{\mathcal T(s)}$ is covered with $\bm B = \bm I$.

\begin{proof}

    First observe that if a meromorphic $\Te_e(s)$ has constant rank $r$ on a contour, it has constant rank on any infinitesimal deformation of the contour. A matrix of rank $r$ has a minor of order $r$ with non-zero determinant, and the determinants of all minors of order larger than $r$ are zero. As the minors are meromorphic functions, they are either identically zero, or their zeros are isolated points. Thus the rank can only change at isolated points of the meromorphic function. As the rank is constant on the contour, none of these points can be on the contour and we can deform the contour avoiding these points.

    Take an $\epsilon < \epsilon^*$ such that for all $\epsilon' \leq \epsilon$, the imaginary axis with $\epsilon'$ indentation at $j \Omega$ does not hit a rank changing point of any $\Te_e(s)$, $e\in\{1,...,E\}$.
    
    By assumption, for all $e\in\{1,...,E\}$, 
    $\Te_e(s)$ is semi-sectorial and has constant rank on this $\epsilon$-indented imaginary axis (contour).

    Combining \eqref{eq:sub-field}, \eqref{eq:conv-hull}, and \eqref{eq:assumption1}, semi-sectoriality of $\Te_1(s),...,\Te_E(s)$ implies semi-sectoriality of $\XTX{s}$, for $s\in j\R$.

    Furthermore, by assumption, $\XTX{s}$ has constant rank along the $\epsilon$-indented imaginary axis. 

    Altogether, the above implies that $\XTX{s}$ is frequency-wise semi-sectorial, which concludes the proof.
\end{proof}

\subsection{Proof of Proposition~\ref{thm:small_phase_theorem_block}}

\begin{proof}
By Lemma~\ref{lem:stability_block}, $\bm H  = \bigoplus_n{\bm T}_n$ is stable.
By Lemma~\ref{lem:block_sectoriality}, $\bm H$ is frequency-wise sectorial if \eqref{eq:small phase theorem, block structure version, sectoriality condition H} holds. 
By Lemma~\ref{lem:semi_stability_block}, $\bm G = {\bm B}^\dagger\bigoplus_e\Te_e{\bm B}$ is semi-stable.
By Lemma \ref{lem:semi_sectoriality_block}, $\bm G$ is frequency-wise semi-sectorial, if \eqref{eq:small phase theorem, block structure version, semi-sectoriality condition G} and the constant rank condition hold.

Using one more time the convex hull property \eqref{eq:conv-hull}, in particular \eqref{eq:phase bounds block matrix}, and the subset property \eqref{eq:sub-field}, the assumptions \eqref{eq:small phase theorem, block structure version, stability condition sup}-\eqref{eq:small phase theorem, block structure version, stability condition inf}
yield
\begin{align}
    \sup_{s\notin j\Omega}\left[\overline{\phi}\left(\bigoplus_n{\bm T}_n\right) + \overline{\phi}\left({\bm B}^\dagger\bigoplus_e\Te_e{\bm B}\right)\right] &< \pi\, , \\
    \inf_{s\notin j\Omega}\left[\underline{\phi}\left(\bigoplus_n{\bm T}_n\right) + \underline{\phi}\left({\bm B}^\dagger
    \bigoplus_e\Te_e{\bm B}\right)\right] &> -\pi\, ,
\end{align}
where $\Tn_n$ and $\Te_e$ are functions of $s$.
These are the phase conditions \eqref{eq:small phase condition max phases}-\eqref{eq:small phase condition min phases} of Theorem \ref{thm:small-phase}.
All in all, the system $\left(\bigoplus_n{\bm T}_n\right)\#\left({\bm B}^\dagger\bigoplus_e\Te_e{\bm B}\right)$ then satisfies all assumptions and conditions of Theorem~\ref{thm:small-phase} and is therefore stable, which concludes the proof.
\end{proof}

% closes {\appendices
}
%%%%%%%%%%%%%%%%%%%%%%%%%%%%%%%%%%%%%%%%%%%%

\bibliography{zotero}

\newpage

\end{document}